\documentclass[a4paper,11pt]{article}

\usepackage{amssymb}
\usepackage{amsthm}
\usepackage{tikz}
\usepackage{tikz,pgfplots}
\usepackage{tkz-graph}
\usepackage{tkz-berge}
\usetikzlibrary{decorations.markings}
\usepackage{amsmath,amssymb}
\usepackage{todonotes}
\usepackage[T1]{fontenc}
\pgfplotsset{compat=1.18}
\usepackage{cite}
\usepackage{color}
\usepackage{hyperref}

\hypersetup{colorlinks=true}

\hypersetup{colorlinks=true, linkcolor=blue, citecolor=blue,urlcolor=blue}

\usepackage{graphicx}
\usepackage{placeins,caption}

\usepackage[top=3cm,bottom=3cm,left=2.8cm,right=2.8cm]{geometry}
	\newtheorem{theorem}{Theorem}[section]
	\newtheorem{lemma}[theorem]{Lemma}
	\newtheorem{corollary}[theorem]{Corollary}
	\newtheorem{proposition}[theorem]{Proposition}

	\theoremstyle{definition}
	
	\newtheorem{claim}[theorem]{Claim}
        \newtheorem{claim-proof}{Claim}
	
	\newtheorem{remark}[theorem]{Remark}

\begin{document}
	
\title{The $k$-distance mutual-visibility problem in graphs}

\author{Mart\'in Cera L\'opez$^{a}$\thanks{Email: \texttt{mcera@us.es}}
\and
Pedro Garc\'ia-V\'azquez$^{a}$\thanks{Email: \texttt{pgvazquez@us.es}}
\and
Juan Carlos Valenzuela-Tripodoro$^{b}$\thanks{Email: \texttt{jcarlos.valenzuela@uca.es}}
\and
Ismael G. Yero$^{b}$\thanks{Email: \texttt{ismael.gonzalez@uca.es}}
}

\maketitle

\begin{center}
$^a$ Departamento de Matem\'atica Aplicada I, Universidad de Sevilla, Spain \\
	\medskip
$^b$ Departamento de Matem\'aticas, Universidad de C\'adiz, Algeciras Campus, Spain
    \\
	\medskip
\end{center}
		
\begin{abstract}
The concept of mutual visibility in graphs, introduced recently, addresses 
a fundamental problem in Graph Theory concerning the identification of the 
largest set of vertices in a graph such that any two vertices have a shortest 
path connecting them, excluding internal vertices of the set. Originally 
motivated by some challenges in Computer Science related to robot navigation, 
the problem seeks to ensure unobstructed communication channels between 
navigating entities. The mutual-visibility problem involves determining a 
largest mutual-visibility set in a graph. The mutual-visibility number of a graph represents the cardinality of the largest mutual-visibility set. 
This concept has sparked significant research interest, leading to connections with 
classical combinatorial problems like the Zarankiewicz problem and Turán-type problems. 
In this paper, we consider practical limitations in network visibility and our 
investigation extends the original concept to $k$-distance mutual-visibility. In 
this case, a pair of vertices is considered $S$-visible if a shortest path of length 
at most $k$ exists, excluding internal vertices belonging to the set $S$. The 
$k$-distance mutual-visibility number represents the cardinality 
of a largest $k$-distance mutual-visibility set. We initiate the study of this new
graph parameter. We prove that the associate decision problem belongs to the 
NP-complete class. We also give some properties and tight bounds, as well as, the exact value of such parameter for some particular non trivial graph classes.
\end{abstract}


\noindent
AMS Subj.\ Class.\ (2020): 05C12, 05C76

\section{Introduction}

Mutual-visibility in graphs relates to one recently introduced problem in graphs seeking for
the largest set of vertices $S$ in a graph such that for any two vertices $u,v\in S$ there is
some shortest $u,v$-path joining them, and whose internal vertices are not in $S$. The problem
was introduced in \cite{DiStefano}, and somehow originated from some computer science problem
related to a model for robot navigation in networks that requires the existence of a channel
(understood as a shortest path) between any two robots in the network that avoid
other navigating robots. For more details on related navigation models, we
suggest \cite{aljohani-2018a,bhagat-2020,Cicerone-2023+,diluna-2017,poudel-2021}.

Formally, given a connected graph $G$ and a set of vertices $S\subset V(G)$, it is said that
two vertices $x,y\in S$ are $S$-\textit{visible} if there is a shortest $x,y$-path (a geodesic) 
$P$ such that $P\cap S=\{x,y\}$. The set $S$ is called a \textit{mutual-visibility set} of $G$ 
if any two vertices of $S$ are $S$-visible. The cardinality of a largest mutual-visibility set 
of $G$ is the \textit{mutual-visibility number} of $G$, denoted by $\mu(G)$. The 
\textit{mutual-visibility problem} in graphs in that of finding a largest possible 
mutual-visibility set (equivalently, computing the mutual-visibility number). The concepts above
were first presented and studied in \cite{DiStefano}. Following this groundbreaking and recent 
research, a considerable amount of literature has since been produced on the subject. Moreover, 
several challenging problems have emerged, connecting the mutual-visibility problem with 
some classical combinatorial problems; for instance, the Zarankiewicz problem 
(see \cite{Cicerone-2023}) or some Tur\'an type problems (see \cite{Bujtas, Cicerone-2023+b}). 
Other interesting contributions to the problem are \cite{Axenovich-2024,Boruzanli,
Bresar,Cicerone-2023+a,Cicerone-hered,variety-2023,Cicerone-2022+,Korze-2023,kuziak-2023, tian-2023+}.

In this investigation, we consider the fact that ``visibility'' between entities in a network
could be frequently limited (with respect to the length of the geodesics considered) for
several reasons. In this sense, for a given integer $k\ge 1$ and a subset of vertices $S$ 
of a graph $G$, two vertices $x,y\in V(G)$ are called $S_k$-\textit{visible} if there is a
shortest $x,y$-path $P$ of length at most $k$ such that all the internal vertices of $P$ are
not in $S$. The set $S$ is a $k$-\textit{distance mutual-visibility set} if every two vertices
of $S$ are $S_k$-visible. Furthermore, the cardinality of a larger $k$-distance mutual-visibility
set in $G$ is the $k$-\textit{distance mutual-visibility number} of $G$, which is denoted
$\mu_k(G)$. Similarly as with the original concept, the $k$-\textit{distance mutual-visibility 
problem} stands for the problem of finding a largest possible $k$-distance mutual-visibility 
set (equivalently, computing the $k$-distance mutual-visibility number). In connection with
these concepts, given a graph $G$, we readily observe the following details:
\begin{itemize}
    \item For every $k\ge 1$, any $k$-distance mutual-visibility set is also a $(k+1)$-distance
    mutual-visibility set.
    \item The case $k=diam(G)$ coincides with the classical mutual-visibility concepts, where 
    $diam(G)$ represents the \textit{diameter} of $G$ (the length of a largest geodesic in $G$).
    \item If $k=1$, then a $1$-distance mutual-visibility set is precisely a clique of $G$ and 
    vice versa. Thus, the $1$-distance mutual-visibility number of $G$ is the same as the
    \textit{clique number} $\omega(G)$ of $G$, which stands for the cardinality of a largest set 
    of vertices inducing a complete graph.
    \item If $G$ is a graph with diameter $2$, then the $k$-distance mutual-visibility problem 
    can only be studied for $k=1$ and $k=2=diam(G)$. In this context, as both situations are already 
    established, our study will concentrate specifically on graphs with a diameter of at least $3$.
\end{itemize}
As a consequence of some of these facts, the following observation follows.

\begin{remark}\label{cadena}
For any graph $G$ of diameter $d$,
$$\omega(G)=\mu_1(G)\le \mu_2(G)\le \cdots \le \mu_{d}(G)=\mu(G).$$
\end{remark}

It is now our goal to present several contributions on the $k$-distance mutual-visibility 
problem in graphs of diameter at least three for any $k$ such that $2\le k\le diam(G)-1$.  

In Section 2 we prove some basic results. Section 3 is devoted to the study of the 
algorithmic complexity of the $k$-distance mutual-visibility problem. Next, in Section 4
we give the exact value of this new parameter, $\mu_k$, for several classes of graphs. 
Finally, we close this work in Section 5 by giving some general bounds in terms of 
structural parameters of the graph like the girth or the minimum and maximum degree.

\subsection{Some terminologies and notations}

Throughout our exposition, we shall use some extra notations that are next stated. 
For instance, we shall write $[n]=\{1,2,\dots, n\}$.  All the graphs considered are 
not directed, connected, simple, and without loops and multiple edges. The \textit{order} 
of $G=(V(G),E(G))$ is $n=|V(G)|$ and the \textit{size} is $m=|E(G)|$. For a given 
vertex $v\in V(G)$, we use $N_G(v)$, $N_G[v]=N_G(v)\cup\{v\}$, and $d_G(v)=|N_G(v)|$ 
as the \textit{open neighborhood}, the \textit{closed neighborhood}, and the \textit{degree} 
of $v$, respectively. Two vertices $u,v$ are called (true or false, resp.) twins in $G$ 
if ($N_G[u]=N_G[v]$ or $N_G(u)=N_G(v)$, resp.). Also, a vertex $u$ is called an 
\textit{extreme vertex} if $N_G[u]$ induces a complete graph.

We denote by $L(G)=\{x\in V(G)\ : \ d_G(v)=1\}$ the set 
of \textit{leaves} of $G$ and 
by $S(G)=\{x\in V(G): \text{ there exists } y\in L(G)\cap N_G(x)\}$ the set of \textit{support} 
vertices of $G$. A tree $T$ is called \textit{complete} if for all $i\in L(T)$ there exists 
$j\in L(T)$ such that $d_G(i,j)=diam(T)$. We define a \textit{perfect tree} as a complete tree in which all vertices, except the leaves, have the same degree.

Given a subgraph $H$ of $G,$ we denote $N_G(H)=\langle N_G(V(H))\rangle_G$ and 
${N}_G^r(H)=\langle \{i\in V(G)\ : \ d(i,H)\le r\} \rangle_G$ for $r\ge 2$ a positive integer.  
The \textit{girth} of $G$ represents the length of a shortest cycle in $G$. Also, if $G$ is a 
graph of diameter $d$ and $k\in [d]$ is an integer, then by $\mathcal{G}^k$ we represent 
the family of all induced isometric subgraphs of $G$ of diameter $k$. 

Given two graphs $G$ and $H$, the \textit{Cartesian} and \textit{strong product} of $G$ 
and $H$ are the graphs $G\Box H$ and $G\boxtimes H$, respectively, defined as follows.
\begin{itemize}
    \item Both $G\Box H$ and $G\boxtimes H$ have vertex set $V(G)\times V(H)$.
    \item The vertices $(g,h),(g',h')\in V(G)\times V(H)$ are adjacent in $G\Box H$ 
    if either $g=g'$ and $hh'\in E(H)$, or $h=h'$ and $gg'\in E(G)$.
    \item The vertices $(g,h),(g',h')\in V(G)\times V(H)$ are adjacent in $G\boxtimes H$ 
    if either they are adjacent in $G\Box H$, or $hh'\in E(H)$ and $gg'\in E(G)$.
\end{itemize}

The \emph{corona graph} $G\odot H$ of any two graphs $G$ and $H$, is the graph obtained 
from one copy of $G$ and $|V(G)|$ copies of $H$, by joining the $i$th vertex of $G$ 
to every vertex in the $i$th copy of $H$.

\section{Some basic results}

Given a graph $G$ of diameter $d$ and an integer $k$ with $1\le k\le d$, the definitions 
of the $k$-distance mutual-visibility concepts might suggest to consider studying ``only'' 
the subgraphs of $G$ of diameter $k$ so that in such subgraphs, we are ``only'' required 
to consider the classical mutual-visibility problem. However, although there exists some 
relationship, it is not exactly clear how this can be further used. We next describe several 
arguments to support our statement. To begin with, we first present a relationship between 
our parameter $\mu_k(G)$ in a graph $G$ and the classical mutual-visibility numbers of some 
induced subgraphs of $G$ of diameter $k$, but that it is only a result ``in one direction''.

\begin{theorem}
\label{th:subgraphs}
Let $G$ be a graph of diameter $d$. Then for any $k\in [d]$,
$$\mu_k(G)\ge \max\{\mu(H)\,:\,H\in\mathcal{G}^k\}.$$
\end{theorem}

\begin{proof}
Let $H\in \mathcal{G}^k$ such that $\mu(H)=\max\{\mu(H')\,:\,H'\in\mathcal{G}^k\}$, and 
let $S_H$ be a $\mu$-set of $H$. Since $H$ has diameter $k$ and is an isometric subgraph of $G$, 
it clearly follows that $S_H$ is also a $k$-distance mutual-visibility set of $G$. Thus, 
$\mu_k(G)\ge|S_H|=\max\{\mu(H)\,:\,H\in\mathcal{G}^k\}$.
\end{proof}

The equality in the bound above does not always happen. To see this consider, for instance, a cycle 
$C_{2r}$ with $r\ge 6$, which has diameter $r$. If $2r/3\le k\le r-1$, then every subgraph of $C_{2r}$ 
of diameter $k$ is a path, whose mutual-visibility number is $2$. However, one can easily see that a 
set of three vertices of $C_{2r}$ taken in such a way that the distances between any two of them are 
as equal as possible (roughly about $2r/3$) is a $k$-distance mutual-visibility set of $C_{2r}$.



We now focus on the inequality chain from Remark \ref{cadena}, in order to show several situations that 
can happen regarding the relationships between $\mu_k(G)$, $\mu(G)$, and the values of $k$. We first show 
that such a chain can be strict in all cases. To do so, we consider the strong product graph $P_r\boxtimes P_2$. 

\begin{proposition}
\label{prop:strong-P_r-P_2}
For any integer $r\ge 2$ and any $2\le k\le r-1$, $\mu_k(P_r\boxtimes P_2)=k+3$.
\end{proposition}

\begin{proof}
Assume $V(P_n)=[n]$ for any integer $n$. Let $S=\{(1,1),(2,1),\dots,(k+1,1)\}\cup\{(1,2),(k+1,2)\}$. It 
can be readily observed that the vertices of $S$ are pairwise at distance at most $k$. Moreover, for any 
two of such vertices, there is a geodesic joining them avoiding the remaining vertices of $S$. Thus, each 
two vertices of $S$ are $S_k$-visible, and so, $S$ is a $k$-mutual-visibility set of $P_r\boxtimes P_2$ 
of cardinality $k+3$.

On the other hand, let $S'$ be a $\mu_k$-set of $P_r\boxtimes P_2$. Let $i\in [r]$ be the smallest integer
with $(i,j)\in S'$ for some $j\in [2]$. Hence, the largest integer $i'\in [r]$ such that $(i',j')\in S'$ 
for some $j'\in [2]$ is at most $i+k$ (namely $i'\le i+k$). In addition, we observe the following two facts.
\begin{itemize}
    \item $|S'\cap \{(i,1),(i,2)\}|\le 2$ and $|S'\cap \{(i',1),(i',2)\}|\le 2$.
    \item For every $q\in \{i+1,\dots,i'-1\}$, it holds $|S'\cap \{(q,1),(q,2)\}|\le 1$. For otherwise, any 
    vertex $x\in S'\cap \{(i,1),(i,2)\}$ and any vertex $y\in S'\cap \{(i',1),(i',2)\}$ are not $S_k$-visible, 
    which is not possible.
\end{itemize}
As a consequence of the items above, since $i'\le i+k$, we deduce that 
$$\mu_k(P_r\boxtimes P_2)=|S'_k|\le 4+\sum_{q=i+1}^{q=i'-1} |S'\cap \{(q,1),(q,2)\}|\le 4+i'-1-i\le k+3.$$
Therefore, the desired equality follows.
\end{proof}

By using the result above, we see that there is a strict inequality in each of the positions from the chain 
of Remark \ref{cadena}. That is, for any $r\ge 2$ and any $1\le k\le r-1$, it holds
$$4=\omega(P_r\boxtimes P_2)=\mu_1(P_r\boxtimes P_2) < 5 = \mu_2(P_r\boxtimes P_2) <  
\cdots < r+2 = \mu_{d}(G)=\mu(P_r\boxtimes P_2).$$

The difference between each pair $\mu_k(G),\mu_{k+1}(G)$ in the inequality above is exactly one. In addition, 
note that the difference between any two of such pairs can be very large, as well. To see this, we only need
to consider the strong product graph $P_r\boxtimes K_s$ for any $s\ge 3$. Similar arguments as the ones used 
in Proposition \ref{prop:strong-P_r-P_2} allow us to prove the following result, which we include without proof.

\begin{proposition}
\label{prop:strong-P_r-K_s}
For any integers $r,s\ge 2$ and any integer $k$ such that $2\le k\le r-1$, $\mu_k(P_r\boxtimes K_s)=(k+1)(s-1)+2$.
\end{proposition}

In contrast to the comments above, there are also graphs for which there is an equality in all the values in 
the chain of Remark \ref{cadena}. To this end, note that a graph $G$ satisfying this fact must hold that 
$\mu(G)=\omega(G)$. It was proved in \cite{DiStefano} that if $G$ is a block graph with vertex set $V(G)=C\cup E$, 
where $C$ is the set of cut vertices and $E$ is the set of extreme vertices of $G$, then $\mu(G)=|E|$. 

In this sense, consider now $G_t$ is a block graph obtained in the following way. We begin with a complete 
graph $K_t$, $t\ge 3$, with vertex set $V(K_t)=[t]$. Next we add $t$ disjoint paths of order $r\ge 2$, 
say $P_r^{(1)}$, $\dots$, $P_r^{(t)}$. Finally, to obtain $G_t$, we add an edge between the vertex
$i\in V(K_t)$ and one leaf of the path $P_r^{i}$ (in some places such graphs are called sun graphs or 
bright sun graphs). It can be observed that $G_t$ is a block graph with $t$ extreme vertices and clique number 
$\omega(G_t)=t$. Moreover, $G_t$ has diameter $2r-1$, which can be as large as desired. Hence, by
using the results from \cite{DiStefano} we deduce that  $\mu(G_t)=t=\omega(G_t)$. 

By using these facts, we see that for any graph $G_t$, with $t\ge 3$ it holds, 
$$t=\omega(G)=\mu_1(G)=\mu_2(G)=\cdots =\mu_{d}(G)=\mu(G).$$

We next include some other property of the $k$-distance mutual-visibility sets of graphs that shall 
be later useful, now regarding the case of extreme vertices of the graph.

\begin{remark}
\label{rem:extreme}
Let $G$ be a graph, let $u,v$ are two adjacent extreme vertices, and let $S\subset V(G)$ be a $k$-mutual 
visibility set of the largest cardinality in $G$. Then $u\in S$ if and only if $v\in S$.
\end{remark}

\section{Complexity of the {\it k}-distance mutual-visibility problem}\label{sec:complex}

This section is focused on showing some complexity issues regarding the problem of finding the 
$k$-distance mutual-visibility number of graphs. That is, the following decision problem.

\begin{center}
\fbox{
	\parbox{0.9\textwidth}{
		{\sc$k$-distance mutual-visibility problem} \\
		\textit{Instance}: A graph $G$ of diameter $d\ge 3$, an integer $1\le k\leq d$, and 
                an integer $R$.\\
		\textit{Question}: Does $G$ has a $k$-distance mutual-visibility set of cardinality at least $R$?}}
\end{center}

To prove the complexity of the problem above we make use of the classical NP-complete problem 
regarding the independence number $\alpha(G)$ of a graph $G$, i.e., the cardinality of a largest
set of vertices of $G$ inducing an edgeless graph.

\begin{center}
\fbox{
	\parbox{0.6\textwidth}{
		{\sc independent set problem} \\
		\textit{Instance}: A graph $G$ and an integer $Q\ge 1$.\\
		\textit{Question}: Is it true that $\alpha(G)\geq Q$?}}
\end{center}

\begin{theorem}
For any integer $1\le k\leq d$, the {\sc $k$-distance mutual-visibility problem} is NP-complete for 
an arbitrary graph $G$ of order $n\ge 3$, size $m$ and diameter $d\ge 3$.
\end{theorem}

\begin{proof}
We first observe that the {\sc $k$-distance mutual-visibility problem} belongs to the NP class since it is 
polynomially possible to check on whether a given set of vertices of cardinality at least $R$ is a $k$-distance 
mutual-visibility set. Moreover, if $k=1$ ($k=d$ resp.), then as already mentioned, the $k$-distance 
mutual-visibility of a graph $G$ is indeed the clique number (mutual visibility, resp.) of $G$, whose NP-complexity 
is already known. In this sense, from now on we assume $G$ is a graph of diameter at least $3$ and that $d-1\ge k \ge 2$.

To describe a polynomial reduction from the {\sc independent set problem}, we begin with an arbitrary graph $G$ 
of order $n$, diameter $d$ and size $m$. 
Assume $V(G)=[n]$ and consider a new graph $G'$ constructed in the following way.
\begin{itemize}
  \item For each edge $e=ij$ of $G$, we add a clique of order $n$ with set of vertices $A_e=A_{ij}$ and the 
  edges $ix$ and $jx$ for every $x\in A_{e}$.
  \item We add an isolated vertex $w$, and all the possible edges $iw$, $xw$ for every $i\in V(G)$, and every 
  $x\in A_e$ with $e\in E(G)$.
  \item We add a set of $d-1$ disjoint cliques of order $n$ with set of
  vertices $B_1, \ldots, B_{d-1}$.
  \item We add all possible edges $ix$ for every $i\in V(G)$ and every $x\in B_1$.
  \item For any $i\in [d-2]$, we add all the possible edges $xy$ for every $x\in B_i$ and every $y\in B_{i+1}$.
\end{itemize}

Notice that the new graph $G'$ has diameter $d$ and that it can be constructed in polynomial time concerning 
the order of $G$. An example of the graph $G'$, for $G = P_5$, is given in Fig.~\ref{fig:G'}.

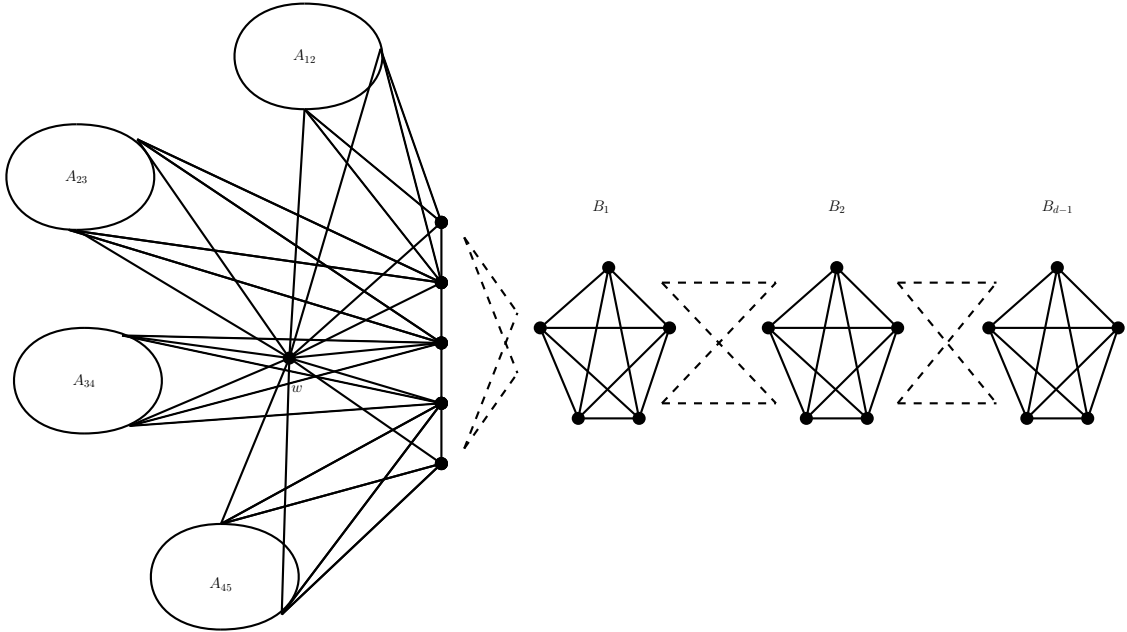
\begin{figure}[h]
\centering
\begin{tikzpicture}[scale=.4, transform shape, style=thick]
\tikzstyle{none}=[fill=white, draw=white, shape=circle]
\tikzstyle{black_V}=[fill=black, draw=black, shape=circle]
\tikzstyle{red_V}=[fill=red, draw=red, shape=circle]
\tikzstyle{blue_V}=[fill=blue, draw=blue, shape=circle]

\tikzstyle{red_E}=[-, draw=black, fill=black, thin]
\tikzstyle{dashed_E}=[-, dashed]
\tikzstyle{green_E}=[-, draw=black]
\tikzstyle{magenta_E}=[-, draw=black]
		\node [style={black_V}] (0) at (-1, 5) {};
		\node [style={black_V}] (1) at (-1, 3) {};
		\node [style={black_V}] (2) at (-1, 1) {};
		\node [style={black_V}] (3) at (-1, -1) {};
		\node [style={black_V}] (4) at (-1, -3) {};
		\node [style=none] (6) at (-5.5, 12.25) {};
		\node [style=none] (8) at (-5.5, 8.75) {};
		\node [style=none] (10) at (-5.5, 10.5) {\Large $A_{12}$};
		\node [style=none] (17) at (-8.25, -5) {};
		\node [style=none] (18) at (-8.25, -8.5) {};
		\node [style=none] (20) at (-12.75, 1.5) {};
		\node [style=none] (21) at (-12.75, -2) {};
		\node [style=none] (22) at (-12.75, -0.25) {\Large $A_{34}$};
		\node [style=none] (23) at (-13, 8.25) {};
		\node [style=none] (24) at (-13, 4.75) {};
		\node [style=none] (25) at (-13, 6.5) {\Large $A_{23}$};
		\node [style=none] (26) at (-6, 0.5) {};
		\node [style=none] (37) at (-8.25, -7) {\Large $A_{45}$};
		\node [style={black_V}] (92) at (4.5, 3.5) {};
		\node [style={black_V}] (93) at (2.25, 1.5) {};
		\node [style={black_V}] (94) at (3.5, -1.5) {};
		\node [style={black_V}] (96) at (5.5, -1.5) {};
		\node [style={black_V}] (97) at (6.5, 1.5) {};
		\node [style=none] (98) at (-5.75, -0.5) {\Large $w$};
		\node [style={black_V}] (99) at (-6, 0.5) {};
		\node [style={black_V}] (100) at (12, 3.5) {};
		\node [style={black_V}] (101) at (9.75, 1.5) {};
		\node [style={black_V}] (102) at (11, -1.5) {};
		\node [style={black_V}] (103) at (13, -1.5) {};
		\node [style={black_V}] (104) at (14, 1.5) {};
		\node [style={black_V}] (105) at (19.25, 3.5) {};
		\node [style={black_V}] (106) at (17, 1.5) {};
		\node [style={black_V}] (107) at (18.25, -1.5) {};
		\node [style={black_V}] (108) at (20.25, -1.5) {};
		\node [style={black_V}] (109) at (21.25, 1.5) {};
		\node [style={black_V}] (110) at (-1, 5) {};
		\node [style={black_V}] (111) at (-1, 3) {};
		\node [style={black_V}] (112) at (-1, 1) {};
		\node [style={black_V}] (113) at (-1, -1) {};
		\node [style={black_V}] (114) at (-1, -3) {};
		\node [style=none] (115) at (6.25, 3) {};
		\node [style=none] (116) at (10, 3) {};
		\node [style=none] (117) at (6.25, -1) {};
		\node [style=none] (118) at (10, -1) {};
		\node [style=none] (119) at (14, 3) {};
		\node [style=none] (120) at (17.25, 3) {};
		\node [style=none] (121) at (14, -1) {};
		\node [style=none] (122) at (17.25, -1) {};
		\node [style=none] (123) at (-0.25, 4.5) {};
		\node [style=none] (124) at (1.5, 2) {};
		\node [style=none] (125) at (-0.25, -2.5) {};
		\node [style=none] (126) at (1.5, 0) {};
		\node [style=none] (127) at (-3, 10.75) {};
		\node [style=black_V] (128) at (-1, 5) {};
		\node [style=none] (129) at (-5.5, 8.75) {};
		\node [style=black_V] (130) at (-1, 3) {};
		\node [style=none] (135) at (-8.25, -5) {};
		\node [style=black_V] (136) at (-1, -1) {};
		\node [style=none] (137) at (-6.25, -8) {};
		\node [style=black_V] (138) at (-1, -3) {};
		\node [style=none] (141) at (-11, 7.75) {};
		\node [style=black_V] (142) at (-1, 3) {};
		\node [style=none] (143) at (-13.25, 4.75) {};
		\node [style=black_V] (144) at (-1, 1) {};
		\node [style=none] (147) at (-11.5, 1.25) {};
		\node [style=black_V] (148) at (-1, 1) {};
		\node [style=none] (149) at (-11.25, -1.75) {};
		\node [style=black_V] (150) at (-1, -1) {};
		\node [style=none] (151) at (4.25, 5.5) {\Large $B_1$};
		\node [style=none] (152) at (12, 5.5) {\Large $B_2$};
		\node [style=none] (153) at (19.25, 5.5) {\Large $B_{d-1}$};
	
		\draw (0) to (1);
		\draw (3) to (2);
		\draw (3) to (4);
		\draw [bend right=90, looseness=2.25] (6.center) to (8.center);
		\draw [bend left=90, looseness=2.50] (6.center) to (8.center);
		\draw [bend right=90, looseness=2.25] (17.center) to (18.center);
		\draw [bend left=90, looseness=2.50] (17.center) to (18.center);
		\draw [bend right=90, looseness=2.25] (20.center) to (21.center);
		\draw [bend left=90, looseness=2.50] (20.center) to (21.center);
		\draw [bend right=90, looseness=2.25] (23.center) to (24.center);
		\draw [bend left=90, looseness=2.50] (23.center) to (24.center);
		\draw (94) to (97);
		\draw (97) to (96);
		\draw (96) to (94);
		\draw (94) to (92);
		\draw (92) to (93);
		\draw (93) to (96);
		\draw (93) to (97);
		\draw (93) to (94);
		\draw (92) to (97);
		\draw (92) to (96);
		\draw (1) to (2);
		\draw (102) to (104);
		\draw (104) to (103);
		\draw (103) to (102);
		\draw (102) to (100);
		\draw (100) to (101);
		\draw (101) to (103);
		\draw (101) to (104);
		\draw (101) to (102);
		\draw (100) to (104);
		\draw (100) to (103);
		\draw (107) to (109);
		\draw (109) to (108);
		\draw (108) to (107);
		\draw (107) to (105);
		\draw (105) to (106);
		\draw (106) to (108);
		\draw (106) to (109);
		\draw (106) to (107);
		\draw (105) to (109);
		\draw (105) to (108);
		\draw [style={dashed_E}] (115.center) to (116.center);
		\draw [style={dashed_E}] (115.center) to (118.center);
		\draw [style={dashed_E}] (117.center) to (118.center);
		\draw [style={dashed_E}] (117.center) to (116.center);
		\draw [style={dashed_E}] (119.center) to (120.center);
		\draw [style={dashed_E}] (119.center) to (122.center);
		\draw [style={dashed_E}] (121.center) to (122.center);
		\draw [style={dashed_E}] (121.center) to (120.center);
		\draw [style={dashed_E}] (123.center) to (124.center);
		\draw [style={dashed_E}] (123.center) to (126.center);
		\draw [style={dashed_E}] (125.center) to (126.center);
		\draw [style={dashed_E}] (125.center) to (124.center);
		\draw [style={green_E}] (127.center) to (128.center);
		\draw [style={green_E}] (127.center) to (130.center);
		\draw [style={green_E}] (129.center) to (130.center);
		\draw [style={green_E}] (129.center) to (128.center);
		\draw [style={green_E}] (135.center) to (136.center);
		\draw [style={green_E}] (135.center) to (138.center);
		\draw [style={green_E}] (137.center) to (138.center);
		\draw [style={green_E}] (137.center) to (136.center);
		\draw [style={green_E}] (141.center) to (142.center);
		\draw [style={green_E}] (141.center) to (144.center);
		\draw [style={green_E}] (143.center) to (144.center);
		\draw [style={green_E}] (143.center) to (142.center);
		\draw [style={green_E}] (147.center) to (148.center);
		\draw [style={green_E}] (147.center) to (150.center);
		\draw [style={green_E}] (149.center) to (150.center);
		\draw [style={green_E}] (149.center) to (148.center);
		\draw [style={green_E}] (111) to (99);
		\draw [style={green_E}] (99) to (110);
		\draw [style={green_E}] (112) to (99);
		\draw [style={green_E}] (99) to (113);
		\draw [style={green_E}] (114) to (99);
		\draw [style={green_E}] (99) to (135.center);
		\draw [style={green_E}] (137.center) to (99);
		\draw [style={green_E}] (99) to (147.center);
		\draw [style={green_E}] (99) to (24.center);
		\draw [style={green_E}] (99) to (141.center);
		\draw [style={green_E}] (99) to (149.center);
		\draw [style={green_E}] (99) to (129.center);
		\draw [style={green_E}] (127.center) to (99);
		\draw [style={red_E}] (127.center) to (128.center);
		\draw [style={red_E}] (142.center) to (129.center);
		\draw [style={red_E}] (129.center) to (128.center);
		\draw [style={red_E}] (142.center) to (127.center);
		\draw [style={red_E}] (148.center) to (147.center);
		\draw [style={red_E}] (147.center) to (150.center);
		\draw [style={red_E}] (150.center) to (149.center);
		\draw [style={red_E}] (149.center) to (148.center);
		\draw [style={magenta_E}] (141.center) to (142.center);
		\draw [style={magenta_E}] (142.center) to (143.center);
		\draw [style={magenta_E}] (143.center) to (148.center);
		\draw [style={magenta_E}] (148.center) to (141.center);
		\draw [style={magenta_E}] (150.center) to (135.center);
		\draw [style={magenta_E}] (135.center) to (138.center);
		\draw [style={magenta_E}] (138.center) to (137.center);
		\draw [style={magenta_E}] (137.center) to (150.center);
	
\end{tikzpicture}
\caption{Construction of the graph $G'$ from $G=P_5$ that has diameter $d=4$.}\label{fig:G'}
\end{figure}

We first consider $k$ as an integer such that $2\le k\le d-1$. Assume $I$ is an independent set of $G$ of the 
largest cardinality, \emph{i.e.},  $|I|=\alpha(G)$. Let $b_j\in B_j$ be a fixed vertex of each $B_j$ with $j\in [d-1]$.
We claim that the set $S$ given next is a $k$-distance mutual-visibility set of $G'$ (we may recall that the 
sets of the second part in the union do not exist when $k=2$).
$$ 
 S=I \cup \left(\bigcup_{e\in E(G)} A_e \right)\cup \left(\bigcup_{j\in [k-2]} (B_j\setminus\{b_j\}) \right)\cup B_{k-1}.
$$

\begin{claim-proof} The following straightforward properties of $S$ and $G'$ hold.
\label{cl:proof}
\begin{itemize}
    \item[(a)] The set $S$ has cardinality $n(m+k-1)+\alpha(G)-k+2.$
    \item[(b)] The vertices in $S$ are pairwise at distance at most $k.$
    \item[(c)] The vertex $w$ does not belong to any geodesic joining a vertex of $\bigcup A_{ij}$ with a 
    vertex of $\bigcup B_j$.
    \item[(d)] The geodesics joining two vertices $x,y\in \bigcup A_{ij}$ are either $xy$, $xwy$, or 
    either $xiy$ with $i\in V(G)$.
\end{itemize}
\end{claim-proof}

Now, to prove our claim (that the set $S$ given above is a $k$-distance mutual-visibility set of $G'$), we 
consider all the possible pairs of vertices $u,v\in S$ and prove that each of these pairs is $S_k$-visible.
Clearly, if $u,v$ are adjacent, then they are $S_k$-visible, and so, from now on we only remark the situations
where $u$ and $v$ are not adjacent. To prove the next situations, we make use of Claim \ref{cl:proof}.

\begin{enumerate}
\item If $u\in A_e$ and $v\in A_{f}$ for two edges $e,f\in E(G)$, then $u,v$ are $S_k$-visible through the geodesic $u w v$.
\item If $u\in A_e$, $v\in V(G)$ and $v\notin e$, then again $u,v$ are $S_k$-visible through the geodesic $u w v$.
\item If $u\in A_e=A_{ij}$ and $v\in B_r$ for some $r\in [k-1]$, then since $I$ is an independent set, at least 
one of the vertices $i$ or $j$ of the edge $e=ij$, say w.l.g. $i$, is not in $I$. If $r=1$, then $u,v$ are 
$S_k$-visible through the geodesic $uiv$. Also, if $r\ge 2$, then $u,v$ are $S_k$-visible through the 
geodesic $u i b_1\cdots b_{r-1}v$.
\item If $u,v\in V(G)$, then $u,v$ are $S_k$-visible through the geodesic $u w v$ (notice that $u,v$ are 
not adjacent according to the definition of $S$).
\item If $u\in V(G)$ and $v\in B_r$ for some $r\in [k-1]$, then $u,v$ are $S_k$-visible through the geodesic 
$u b_1\cdots b_{r-1}v$ (notice that the case $v\in B_1$ is already considered since in such situation $u,v$ would be adjacent).
\item If $u\in B_r$ and $v\in B_s$ for some $r,s\in [k-1]$, with $r<s,$ 
then $u,v$ are $S_k$-visible through the geodesic $u b_{r+1}\cdots b_{s-1}v$.
\end{enumerate}

The described situations include all possible pairs of not adjacent vertices of $G'$ showing that such $S$ is 
a $k$-distance mutual-visibility set of $G'$. Since $|S|=n(m+k-1)+\alpha(G)-k+2,$ we deduce that
\begin{equation}\label{eq:compl-1}
\mu_k(G')\ge n(m+k-1)+\alpha(G)-k+2.
\end{equation}

On the other hand, let $S'$ be a $\mu_k$-set of $G'$ of cardinality
$|S'|\ge n(m+k-1)+\alpha(G)-k+2.$ 
\begin{claim-proof} The following straightforward properties of $S'$ hold.
\label{cl:proof-2}
\begin{itemize}
  \item[(a)] $S'\cap \left( \bigcup A_{ij}  \right) \neq \emptyset$.: That is, since $m>d$ and 
  $2\le k \le d-1 \le n-2$, it holds $|S'| \ge nm+\alpha+(k-1)n-k+2 > nd+1+n-(n-2)+2=nd+5$, and therefore we deduce the conclusion. 
  \item[(b)] $\left( \bigcup A_{ij}\cup \{w\} \cup V(G) \right) \not\subseteq S'$. Suppose to the contrary   
  that $\left( \bigcup A_{ij}\cup \{w\} \cup V(G) \right) \subseteq S'$. Since $n\ge 4$ and $diam(G)\ge 3,$   
  there exist $e,f \in E(G)$ with $e\neq f$. Let $x\in A_e$ and $y\in A_f.$ If $e=ij$ and $f=ih$ are incident 
  edges on the vertex $i\in V(G),$ then by Claim \ref{cl:proof} (d) the geodesics joining the vertices $x$ and   
  $y$ are $xwy$ or $xiy,$ but this is not possible since $\{w\} \cup V(G) \subseteq S'.$ If $e$ and $f$ are not 
  incident, then the geodesic joining $x$ and $y$ is $xwy$ and this is not possible because $w\in S'$.
  
  \item[(c)] $A_{ij}\subseteq S'.$ If $k=2$, then $|S'|\ge nm+n+1$ and if $3\le k \le d-1 \le n-2$, then 
  $|S'|\ge nm+3+2n-(n-2) \ge nm+n+5.$ In any case, from item (b) we derive that $S' \cap \left(\bigcup B_r \right) \neq \emptyset.$
  
  \item[(d)] $S'\cap A_{ij} = A_{ij}$ for every $ij\in E(G)$. First note that any two vertices belonging to 
  a set $A_{ij}$ for some $ij\in E(G)$ are extreme vertices. Hence, by Remark \ref{rem:extreme}, 
  $S'\cap A_{ij} \neq \emptyset$   implies that  $S'\cap A_{ij} = A_{ij}$. Also, if there is a set $A_e$ (with $e=ij$) 
  such that $A_e\cap S'=\emptyset$, then it must happen that $i,j\in S'$. Otherwise, as $S'\cap \left( \bigcup B_{r} 
  \right) \neq \emptyset,$ we can consider $S''=S' \cup A_{ij}$ which would be a larger $k$-distance mutual-visibility 
  set of $G'$, which is not possible. However, in such case, since $n\ge 3$, it holds that the set 
  $S''=S'\setminus\{i,j\}\cup A_e$ is also a $k$-distance mutual-visibility set of $G'$ of cardinality larger 
  than $\mu_k(G')$, which is again not possible. 
\end{itemize}
\end{claim-proof}

As a consequence of the items above, we deduce that $ \bigcup_{e\in E(G)} A_e \subseteq S'$.  Due to the 
$k$-distance mutual-visibility property in $S'$, the largest index $r$ such that it could happen 
$B_r\cap S'\ne \emptyset$ is at most $k-1$. Thus, for every $r\ge k$, we have that $B_r\cap S'=\emptyset$.

On the other hand, assume $r'$ is the largest index such that $B_{r'}\cap S'\ne \emptyset$. Hence, for 
every $r<r'$, it must happen that $|B_r\cap S'| < |B_r|=n$, for otherwise, the vertices from $B_{r'}$ 
will not be $S_k$-visible with the vertices from $ \bigcup_{e\in E(G)} A_e$, which is not possible. 
Consequently, $\left|S' \cap \left(\bigcup_{r\in [r']} B_r\right)\right|\le nr'-r'+1$.

Now, since $\bigcup_{e\in E(G)} A_e \subseteq S'$ and, by the fact 
that there is at least one set  $B_r$ such that $B_r\cap S'\ne \emptyset$,
it must hold that $S'\cap V(G)$ forms an independent set (and so $|S'\cap V(G)|\le \alpha(G)$). For otherwise,
there will be a pair of adjacent vertices $ij\in S'\cap V(G)$, and so, 
each vertex $x\in A_{ij}$ will not be $S_k$-visible with each vertex 
from every set $B_r$, and this is a contradiction. 

In addition to the comments above, notice also that the vertex $w$ does 
not belong to $S'$. Otherwise, since $G$ has a diameter of at least $3$, 
there will be at least two edges $e,f\in E(G)$ with no common end-vertices. 
Hence, each pair of vertices $x\in A_e$ and $y\in A_f$ are not $S_k$-
visible, and this is not possible.

As a conclusion of all the arguments above we deduce the following
$$ 
\begin{array}{rcl}
  \mu_k(G')= |S'|  & = &  
  |S' \cap V(G)| + \left|S'\cap \left(\bigcup_{ij\in E(G)} A_{ij}\right)\right|+ 
  \left|S'\cap \left(\bigcup_{r\in [d-1]} B_{r}\right) \right| 
  \\[1em]
  & \le &  \alpha(G)+nm+(k-2)(n-1)+n
  \\[1em]
  & \le &  n(m+k-1)+\alpha(G)-k+2.
\end{array}
$$

Therefore, by using \eqref{eq:compl-1}, we deduce the equality $\mu_k(G')=n(m+k-1)+\alpha(G)-k+2$. 
Finally, by making $R=n(m+k-1)+Q-k+2$, for any integer $k$ such that $2\le k\le d-1$, it is readily 
seen that $\mu_k(G')\geq R$ if and only if $\alpha(G)\geq Q$, which completes the reduction of the 
{\sc independent set problem} to our {\sc $k$-distance mutual-visibility problem} for any integer 
$k$ such that $2\le k\le d-1$.
\end{proof}

As a consequence of the reduction above, we obtain that the problem of computing the $k$-distance 
mutual-visibility number of graphs is NP-hard in general. It makes sense then, to consider families 
of graphs in which an efficient solution to this problem can be obtained.

\section{{\it k}-distance mutual-visibility on particular graph classes}

Based on the NP-hardness of computing the $k$-distance mutual-visibility number of graphs, we centre 
our attention in this section into computing such value for several non-trivial families of graphs.

\begin{proposition}
Let $k\ge 1$ and $n\ge 3$ be integers. Then
$$ \begin{array}{cc}
   \mu_k(P_n)=2,  & \mu_k(C_n)= \left\{
        \begin{array}{cl}
           3; & \text{if $n\le 3k$}; \\[.1em]
           2; & \text{otherwise}. 
   \end{array} \right.
\end{array} $$
for all $k$ less than or equal to the diameter of the corresponding graph.
\end{proposition}

\begin{proof}
It is straightforward to verify the result for path graphs by 
considering the inequalities pointed out in Remark \ref{cadena}, 
since the clique number of a path matches the value of $\mu(P_n)=2$.

Let $C_n:v_1v_2\ldots v_n$ be a cycle graph. Assume first that $n\ge 3k+1$. Clearly, $\mu_k(C_n)\ge 2$. 
To prove the other inequality we proceed by contradiction.
Let us suppose that there are three vertices, $\{x_i,x_j,x_l:i<j<l\}$, belonging to a $k$-distance 
mutual-visibility set $P$ in $C_n$. Let $C_{i,l}$ be the subgraph of $C$ (which is a path) with 
end-vertices $x_i,x_l$ passing through $x_j$. Since $x_i,x_j$ are $P_k$-visible (and so are $x_j,x_l$), 
it holds that the length of $C_{i,l}$ is at most $2k$. Therefore, the ``other part of the cycle $C_n$'', 
the subgraph induced by $V(C_n)\setminus V(C_{i,l})\cup \{x_i,x_l\}$ is a path of length at least $k+1$, 
which implies that $x_i,x_l$ are not $P_k$-visible, which is a contradiction. So, any $k$-distance 
mutual-visibility set in $C$ must have cardinality at most $2$, and thus the desired equality follows for $n\ge 3k+1$.

On the other hand, assume now that $n\le 3k$. First, it is not difficult to check that $\mu_k(C_n)\ge 3$ 
by just taking a set of three vertices of $C_n$ being pairwise as equidistant as possible. Also, if a 
$k$-distance mutual-visibility set contains four vertices $\{x_i,x_j,x_l,x_m: i<j<l<m\}$, then any shortest 
path joining $x_i$ with $x_l$ must contain either $x_j$ or $x_m$, which is a contradiction. Therefore, 
$\mu_k(C_n)\le 3$, and the equality follows for  $n\le 3k$.
\end{proof}

We now continue in this section with studying the case of trees. To this end, notice that if $n\ge 2$ 
is an integer and $S_n$ is a star on $n+1$ vertices, then since $S_n$ is a tree with 
$diam(S_n)=2$, from \cite{DiStefano} we have that $\mu_2(S_n)=\mu(S_n)=|{L}(S_n)|=n$. Thus, from 
now on, only trees different from stars are considered.

\begin{proposition} {\label{prop:bistar}}
Let $m\ge n\ge 1$ be integers and let $S_{n,m}$ be a bistar. Then $\mu_2(S_{n,m})=m+1$ 
and $\mu_3(B_{n,m})=n+m.$ 
\end{proposition}

\begin{proof}
Note that $diam(S_{n,m})=3.$ Hence, from \cite{DiStefano}, $\mu_3(S_{n,m})=\mu(S_{n,m})=|{L}(S_{n,m})|=n+m$.   
Now, we prove that $\mu_2(S_{n,m})=m+1$. Let $B(S_{n,m})=\{x,y\}$ where $d_{S_{n,m}}(x)=m+1$ and $d_{S_{n,m}}(y)=n+1.$  
It is easy to check that $S=N_{S_{n,m}}(x)$ is a $2$-distance mutual-visibility set of $S_{n,m}$ with $|S|=m+1$, 
and so $\mu_2(S_{n,m})\ge m+1.$

On the other hand, let $S'$ be a $\mu_2-$set of $S_{n,m}$. If $S'\cap N_{S_{n,m}}(x) = \emptyset,$ 
then $S'\subseteq N_{S_{n,m}}(y)$ and $|S'|\le n+1$. However, since $n+1\le m+1$, it must happen 
that $m=n$, and in such case we indeed have $\mu_2(S_{n,m})=m+1$, as desired. 
If $n<m$, then $S'\cap N_{S_{n,m}}(w_m) \not = \emptyset.$ Given $x\in S'\cap N_{S_{n,m}}(w_m)\setminus\{w_n\}$ 
and $y\in N(w_n)\setminus \{w_m\},$ we have that $d(x,y)>2$ and then $y\not \in S'.$ Therefore 
$S'\subseteq N(w_m)$ and $|S'|\le m+1,$ which conclude the proof. 
\end{proof}

\begin{proposition}
    Let $T$ be a tree and let $k$ be an integer such that
    $2\le k\le diam(T)$. Then
    $$ \mu_k(T)=\max\{ \mu(T_k):T_k\in {\cal T}_k \} $$
    where ${\cal T}_k=\{T_k\subseteq T: T_k 
    \text{ is a tree with diameter at most $k$}\}.$
\end{proposition}

\begin{proof}
Given a tree $T_k\subseteq T$ with diameter at most $k$. Since $T_k$ is a connected induced 
isometric subgraph of $T$ with diameter at most $k$ we have that $\mu_k(T)\ge |V(T_k)|$. Hence, 
$\mu_k(T) \ge \max\{ \mu(T_k):T_k\in {\cal T}_k \}$.
    
On the other hand, let $S$ be a $\mu_k$-set in $T$. The subgraph induced by the set of vertices $S$ 
is connected and therefore it is a subtree $T'\subseteq T$. Since $S$ is a $k$-distance 
mutual-visibility set, the diameter of $T'$ is at most $k$. Therefore, 
$\mu_k(T)=|S|=\mu_k(T') \le \max\{ \mu(T_k):T_k\in {\cal T}_k \}$, which concludes the proof.  
\end{proof}

Given a caterpillar graph $T$, let us note that 
$B(T)=\{v\in V(T)\ : \ d_T(v)>1\}$ induces 
a path with $r=diam(T)-1$ vertices. The two following corollaries are examples of the result above. To show them it is sufficient to compute the number of leaves of any
subtree of $T$ with diameter exactly $k$.

\begin{corollary}
Let $T$ be a caterpillar graph with $r=diam(T)-1\ge 3$ and 
$d_T(v)=q \ge 2$ for all $v\in B(T).$ 
Let $k$ be an integer with $2\leq k \leq diam(T)=r+1.$ 
Then $\mu_k(T)=k(q-2)-q+4.$  
\end{corollary}

\begin{corollary}
Let $T$ be a perfect tree with maximum degree $\Delta.$ Let 
$k$ be a positive integer with $2\le k \le diam(T).$ 
Then 
$$ 
    \mu_k(T) = \left\{\ 
        \begin{array}{ll}
    \Delta\cdot(\Delta-1)^{\frac{k-2}{2} }; 
        & \text{if } k \text{ is even}; \\[0.4em]
    2\cdot (\Delta-1)^{\frac{k-1}{2}};  & \text{if } k \text{ is odd.}
\end{array} 
\right.
$$     
\end{corollary}

To end this section, we now consider the case of corona graphs whose first factor is a path. Note that, given any graph $G$, by reasoning as in Proposition \ref{prop:bistar}, we deduce that $\mu_2(P_2\odot G)=|V(G)|+1$ and that and $\mu_3(P_2\odot G)=2\cdot|V(G)|$. On the other hand, if $|V(G)|=1$, then $P_r\odot G$ is a caterpillar. The remaining related corona graphs are next studied.

\begin{theorem}
\label{th:corona}
Let $G$ be a graph with $n=|V(G)|\ge 2$ and let $P_r$ be the path on $r\ge 3$ vertices. If $k$ is a positive integer with $2\le k\le r+1$, then 
\begin{enumerate}
    \item $\mu_k(P_r\odot G)=(k-1)\cdot n+2$ for any $k$ with $2\le k\le r-1$; 
    \item $\mu_{r}(P_r\odot G)=(r-1)\cdot n+1$; and
    \item $\mu_{r+1}(P_r\odot G)=\mu(P_r\odot G)r\cdot n$.
\end{enumerate}
\end{theorem}

\begin{proof}
Let $2\le k\le r-1$. Assume $V(P_r)=\{x_i\,:\, i\in [r]\}$ such that $x_ix_{i+1}\in E(P_r)$ for $i\in [r-1]$. We denote 
by $G_i$ the $i$th copy of $G$ associated to the $i$th vertex $x_i.$
We have that $S=\{x_1\}\cup V(G_2)\cup\cdots\cup V(G_{k})\cup \{x_{k+1}\}$ is a $k$-distance mutual-visibility set of the corona graph $P_r\odot G$, and therefore $\mu_k(P_r\odot G)\ge |S|=(k-1)\cdot n+2.$

Now, let $S'$ be a $\mu_k$-set of $P_r\odot G.$ If $x\in S'\cap V(G_i)$ for some $i\in [r]$, then we claim that $V(G_i)\subseteq S'$. Suppose to the contrary that there is a vertex $x'\in V(G_i)\setminus S'$. If $x_i\notin S'$, then clearly $\{x'\}\cup S'$ is also a $k$-distance mutual-visibility set of $P_r\odot G$, which is not possible. Moreover, if $x_i\in S'$, then by the structure of corona graphs ($x_i$ is a cut vertex), it must happen that $S'\subseteq V(G_i)\cup \{x_i\}$, which means $|S'|\le n+1$, but $|S'|=\mu_k(P_r\odot G)\ge (k-1)\cdot n+2\ge n+2$, which is again a contradiction. Thus, $V(G_i)\subseteq S'$ as claimed. We consider $V(G_{i^*})$ with $i^*=\min\{i\ | \ S\cap V(G_i)\not=\emptyset \}.$ We note 
that if $i^*=1$ or $i^*=r-k+2$, then we have $|S'|\le (k-1)\cdot n+1,$ and this not possible since $|S'|\ge (k-1)\cdot n+2.$ Therefore, $S'\subseteq \{x_{{i^*}-1}\}\cup V(G_{i^*})\cup \ldots\cup V(G_{k+{i^*}-2})\cup\{x_{k+{i^*}-1}\}$, and so $|S'|\le (k-1)\cdot n+2$, which gives the equality. 

For $k=r,$ the set $S=V(G_1)\cup\ldots\cup V(G_{r-1})\cup \{x_{r}\}$ is an $r$-distance mutual-visibility set of cardinality $(r-1)\cdot n+1.$ Let us consider $|S'|$ be a $\mu_k$-set. By reasoning as in the 
previous case, for $1\le i^*\le 2$ we have that $|S'|\le (r-1)\cdot |V(G)|+1.$ If $i^*\ge 3,$ then 
$V(G_2)\cap S'=\emptyset.$ Thus, $S''=S'\setminus\{x_2\}\cup V(G_2)$ is a $k$-distance mutual-visibility set with $|S''|>|S'|,$ a contradiction. 

Finally, if $k=r+1=diam(P_r\odot G)$, then $\mu_k(P_r\odot G)=\mu(P_r\odot G)=r\cdot n.$ It is clear that $S=\bigcup_{j=[r]} V(G_j)$ is a $k$-distance mutual-visibility set. Now, given $S'$ a $\mu_k$-set, 
we have that $S'\subseteq \bigcup_{j=[r]} V(G_j).$ Otherwise, if there exists $x=x_i\in S'\setminus 
\bigcup_{j\in [r]} V(G_j),$ then $S'\cap V(G_i)=\emptyset.$ Therefore, $S''=S\setminus\{x_i\}\cup V(G_i)$ 
is a $k$-distance mutual-visibility set with $|S''|>|S'|,$ but this is not possible and we conclude the proof.
\end{proof}

\section{General bounds}

Based again on the NP-hardness of computing the $k$-distance mutual-visibility number of graphs, it would be desirable to develop several tight bounds on this parameter to better approximate its value when a closed formula cannot be reached. We next centre our attention on this direction.

We first present a bound for $\mu_k(G)$ of a graph $G$, in terms of the minimum and maximum degrees, and the girth of $G$. 

\begin{lemma}\label{indepe}
Let $G$ be a graph with girth $g$. Let $S$ be a $k$-distance mutual visibility set of $G$, with $|S|\ge 3$, being $k<\lfloor g/2\rfloor$. Then $S$ is an independent set in $G$. 
\end{lemma}

\begin{proof}
Let $u,v,w$ be three vertices of $S$ and let us consider a shortest $u,v$-path, a shortest $v,w$-path and a shortest $u,w$-path, all of them of length at most $k$, and without internal vertices in $S$. If one of these paths is an edge, then such paths induce a cycle of length at most $2k+1\le 2(\lfloor g/2\rfloor-1)+1\le g-1$, which is not possible. Thus, none of these shortest paths can be an edge, which proves the result.
\end{proof}

In what follows, for an integer $r\ge 2$ and a vertex $v$ of a graph $G$, by $N_r(v)$ we mean the set of vertices of $G$ at distance exactly $r$ from $v$.

\begin{theorem}\label{cota2}
    Let $G$ be a graph with girth $g\ge 6$ and maximum degree $\Delta\ge 3$. Then $\mu_2(G)\le 1+\Delta(\Delta-1)$. 
\end{theorem}

\begin{proof}
Let $S$ be a $2$-distance mutual visibility set of $G$. Clearly, $|S|\ge \Delta\ge 3$. Let us 
consider any vertex $v\in S$ and the set of vertices $N_2(v)$ at distance exactly $2$ from $v$ in $G$. By Lemma~\ref{indepe}, we have that $S$ is an independent set, which implies that 
$S\subseteq \{v\}\cup N_2(v)$. As $G$ has girth at least 6, we have $|N_2(v)|\le |N_G(v)|(\Delta-1)
\le \Delta(\Delta-1)$, yielding that $|S|\le 1+\Delta(\Delta-1)$.
\end{proof}

The upper bound of Theorem~\ref{cota2} is tight when the girth of $G$ is 6. Let $G$ be the $(r,6)$-cage, with $r=3,4,5,6$. We know that $G$ is a balanced bipartite graph and it is the incidence graph of a projective plane, because $r-1$ is a prime power. In these cases we can observe that every two 
vertices of the same class are at distance 2, which means that a largest 2-distance mutual visibility 
set of any of these graphs must have cardinality at least $|V(G)|/2$, which is equal to $1+r(r-1)$, 
for $r=3,4,5,6$.   

\begin{theorem}\label{gen}
Let $G$ be a graph with maximum and minimum degrees $\Delta$ and $\delta\ge 2$, respectively, and girth $g\ge 4$. 
Then for every integer $k\ge 1+2\lfloor (g-1)/3\rfloor$, $$\displaystyle \mu_k(G)\ge (\Delta+\delta-2)
(\delta-1)^{\lfloor (g-1)/3\rfloor-1}.$$
\end{theorem}

\begin{proof}
Denote by $d=\lfloor (g-1)/3\rfloor$. Let $v$ be a vertex of $G$ with degree $\Delta$ and let $u\in N_G(v)$. Consider the set $S=\left(N_{d}(u)\cap N_{d+1}(v)\right)\cup\left(N_{d}(v)\cap N_{d+1}(u)\right) $. We claim that $S$ is a $k$-distance mutual-visibility set of $G$. To see this, we must show that for every pair of vertices $x,y\in S$, there exists a shortest $x,y$-path of length at most $k$ with internal vertices not in $S$. 

If $\{x,y\}\subseteq \left(N_{d}(u)\cap N_{d+1}(v)\right)$, then from the shortest $u,x$-path and the shortest $u,y$-path, we can construct an $x,y$-path $Q$ of length at most $2d$. If $\{x,y\}\subseteq  \left(N_{d}(v)\cap N_{d+1}(u)\right)$, then from the shortest $v,x$-path and the shortest $v,y$-path, both of them of length $d$, we can construct an $x,y$-path $Q$ of length at most $2d$. If $x\in N_{d}(u)\cap N_{d+1}(v)$ and $y\in N_{d}(v)\cap N_{d+1}(u)$, then from the shortest $v,x$-path of length $d+1$ and the shortest $v,y$-path of length $d$, we can construct an $x,y$-path $Q$ of length at most $2d+1$. Analogously, if $x\in N_{d}(v)\cap N_{d+1}(u)$ and $y\in N_{d}(u)\cap N_{d+1}(v)$, then from the shortest $v,x$-path of length $d$ and the shortest $v,y$-path of length $d+1$, we can construct an $x,y$-path $Q$ of length at most $2d+1$. In any way, $2d+1=2\lfloor (g-1)/3\rfloor+1\le k$; that is, we may construct an $x,y$-path $Q$ of length at most $k$ in $G$. 

In addition, if the above $x,y$-path $Q$ of length at most $k$ is a shortest $x,y$-path, then we are done, because the internal vertices from the shortest $v,x$-path and the shortest $v,y$-path are not in $S$. If not, then the distance between $x$ and $y$ must be at most $2d-1$, if either $\{x,y\}\subseteq N_{d}(u)\cap N_{d+1}(v)$, or $\{x,y\}\subseteq \left(N_{d}(v)\cap N_{d+1}(u)\right)$; and at most $2d$ if either $x\in N_{d}(u)\cap N_{d+1}(v)$ and $y\in N_{d}(v)\cap N_{d+1}(u)$, or $x\in N_{d}(v)\cap N_{d+1}(u)$ and $y\in N_{d}(u)\cap N_{d+1}(v)$. Let $P$ be a shortest $x,y$-path and let us proceed by contradiction assuming that some internal vertex $z$ from $P$ belongs to $S$. Denote by $P_x$ the fragment of $P$ joining $x$ and $z$, and by $P_y$ the fragment of $P$ joining $z$ and $y$. Without loss of generality, we may assume that the length of $P_x$ is less than or equal to the length of $P_y$. 

If $\{x,y,z\}\subseteq N_{d}(u)\cap N_{d+1}(v)$, then $P_x$, whose length is at most $d-1$, together with the $u,x$-path and the $u,z$-path, both of them of length $d$, form a cycle of length at most $3d-1\le g-2$.

If $\{x,y,z\}\subseteq N_{d}(v)\cap N_{d+1}(u)$, then $P_x$, whose length is at most $d-1$, together with the $v,x$-path and the $v,z$-path, both of them of length $d$, form a cycle of length at most $3d-1\le g-2$.

If $\{x,y\}\subseteq N_{d}(u)\cap N_{d+1}(v)$ and $z\in N_{d}(v)\cap N_{d+1}(u)$, then $P_x$, whose length is at most $d-1$, together with the $v,x$-path, whose length is $d+1$ and the $v,z$-path, of length $d$, form a cycle of length at most $3d\le g-1$.  The procedure is analogous if $\{x,y\}\subseteq N_{d}(v)\cap N_{d+1}(u)$ and $z\in N_{d}(u)\cap N_{d+1}(v)$.

If $\{x,z\}\subseteq N_{d}(u)\cap N_{d+1}(v)$ and $y\in N_{d}(v)\cap N_{d+1}(u)$, then $P_x$, whose length is at most $d$, together with the $u,x$-path and the $u,z$-path, both of them of length $d$, form a cycle of length at most $3d\le g-1$. The arguments are analogous if $\{x,z\}\subseteq N_{d}(v)\cap N_{d+1}(u)$ and $y\in N_{d}(u)\cap N_{d+1}(v)$.

If $x\in N_{d}(v)\cap N_{d+1}(u)$ and $\{y,z\}\subseteq N_{d}(u)\cap N_{d+1}(v)$, then we have two possible situations. If $P_x$ has length exactly $d$, then $P_y$ also has length $d$, because the length of $P$ is at most $2d$ and the length of $P_x$ is less than or equal to the length of $P_y$. In this case, $P_y$ together with the $u,y$-path and the $u,z$-path, both of them of length $d$, form a cycle of length at most $3d\le g-1$. Also, if $P_x$ has length at most $d-1$, then $P_x$ together with the $v,x$-path, whose length is $d$, and the $v,z$-path, of length $d+1$, form a cycle of length at most $3d\le g-1$.  The reasoning is analogous if $x\in N_{d}(u)\cap N_{d+1}(v)$ and $\{y,z\}\subseteq N_{d}(v)\cap N_{d+1}(u)$. 

Hence, the shortest $x,y$-path $P$ has no internal vertices in $S$, yielding that $S$ is a 
$k$-distance mutual-visibility set of $G$.

Finally, since $G$ has girth $g$ and $d(v)=\Delta$, it is clear that $|N_{j}(v)\cap N_{j+1}(u)|\ge 
(\Delta-1)(\delta-1)^{j-1}$, for $j\in [d]$, and $|N_{j}(u)\cap N_{j+1}(v)|\ge (\delta-1)^j$, for $j\in [d]$, because otherwise, a cycle of length at most $2j\le 2d<g$ appears. Hence, 
$$ \displaystyle \mu_k(G)\ge |S|\ge  (\Delta-1)(\delta-1)^{d-1}+(\delta-1)^{d}=(\Delta+\delta-2)(\delta-1)^{d-1},$$
which completes the proof.
\end{proof}

If one wants to get rid of the minimum and maximum degrees to present bounds for $\mu_k(G)$ of a graph $G$ for specific values of $k$, then one needs to separate the study depending on the parity of the girth of $G$. This is made in the next 
two results.

\begin{theorem}\label{oddg}
Let $G$ be a graph with $n$ vertices and odd girth $g\ge 5$ and maximum degree $\Delta\ge 2$. If $S$ is a $(g-1)/2$-distance mutual visibility set of $G$ and $G[S]$ contains some edge, then $$\displaystyle |S|\le 2+(\Delta-1)^{(g-1)/2}.$$
\end{theorem}
  
\begin{proof}
Let $S$ be a $(g-1)/2$-distance mutual-visibility set of $G$. Observe that there is not a path $P_3=uvw$ in $G$ such that $\{u,v,w\}\subseteq S$, because in such a case, from this path and a shortest $u,w$-path of length at most $(g-1)/2$ with the internal vertices not in $S$, we can find a cycle of length at most $(g+3)/2<g$, because $g\ge 5$. Thus, as the subgraph of $G$ induced by, denoted $G[S]$, contains some edge, we deduce that $G[S]$ is formed by some isolated edges and some isolated vertices.

Consider any edge $uv\in E(G)$, with $\{u,v\}\subseteq S$. Since $S$ is a $(g-1)/2$-distance mutual-visibility set of $G$ and the girth of $G$ is equal to $g$, the every vertex $z\in S\setminus\{u,v\}$ is at distance equal to $(g-1)/2$ from $u$ and $v$. That is, for every $z\in S\setminus\{u,v\}$, there are a $z,u$-path and a $z,v$-path, both of them of length $(g-1)/2$, with the internal vertices in $V(G)\setminus S$. Each of these two paths has $(g-3)/2$ internal vertices, all of them in $V(G)\setminus S$. For every $z\in S\setminus\{u,v\}$, denote by $P_z$ the pair of $(g-3)/2$-tuples of internal vertices of the shortest $z,u$-path and the shortest $z,v$-path.  

Let us consider the subsets $F_u=N_{(g-3)/2}(u)\cap N_{(g-1)/2}(v)$ and $F_v=N_{(g-1)/2}(u)\cap N_{(g-3)/2}(v)$. Assume first that $P_z=P_w$ for two vertices $z,w\in S\setminus\{u,v\}$ and let us denote by $z'=F(u)\cap P_z$ and $w'=F(v)\cap P_z$. Then a cycle $z,z',w,w'$ of length 4 appears, which is a contradiction. Hence, $P_z \neq P_w$ for all pair of different vertices $z,w$ in $S\setminus \{u,v\}.$

\begin{claim}\label{S-2} 
$|S|-2\le \min\{|F_u|(d_G(v)-1),(d_G(u)-1)|F_v|\}$.
\end{claim}

Without loss of generality, we may assume that $|F_u|(dv(v)-1)\le (d_G(u)-1)|F_v|$. We proceed by contradiction supposing that $|S|-2>|F_u|(d_G(v)-1)$. Since $P_z \neq P_w$ for all pair of different vertices $z,w$ in 
$S\setminus \{u,v\}$, the number of pair of $(g-3)/2$-tuples of internal vertices of 
the shortest paths between vertices of $S\setminus \{u,v\}$ and $\{u,v\}$ is larger than $|F_u|(d(v)-1)$. 
This implies that there are two vertices, $w,z\in S\setminus\{u,v\}$, that satisfy these two assertions:

(1) The shortest $w,u$-path and the shortest $z,u$-path have the same internal vertices, which means that there is a vertex $y\in N_G(w)\cap N_G(z)$.

(2) There exists $x\in N_G(v)\setminus\{u\}$ such that $x$ is an internal vertex of the shortest $w,v$-path 
and the shortest $z,v$-path.

From (1) and (2) we can observe that the shortest $w,x$-path and the shortest $z,x$-path, both of them of length $(g-3)/2$ 
and the path with vertices $w,y,z$ form a cycle of length $g-1$, which is a contradiction. This proves the Claim. 

From Claim~\ref{S-2} we deduce that $\displaystyle |S|-2\le (\Delta-1)^{(g-1)/2}$ and the result follows.
\end{proof}

As an example, the upper bound of Theorem~\ref{oddg} is reached for the Petersen graph, which is a 
3-regular graph having girth 5. The independence number of this graph is 5. However, we can get a 2-distance 
mutual visibility set $D$ of cardinality 6 (see the red colored vertices of Figure~\ref{fig:peter}). 
Thus, every largest 2-distance mutual visibility set $S$ of the Petersen graph induces a subgraph with at least one edge. Then, by applying Theorem~\ref{oddg}, we have $|S|\le 6$, yielding that $\mu_2=6$.    

\begin{figure}[h]
\centering
\begin{tikzpicture}[scale=.4, transform shape, style=thick]
\tikzstyle{vert_black}=[fill=black, draw=black, shape=circle]
\tikzstyle{vert_dest}=[fill=none, draw=red, shape=circle, scale=2.25]

		\node [style={vert_black}] (0) at (0, 9) {};
		\node [style={vert_black}] (1) at (-6, 4) {};
		\node [style={vert_black}] (2) at (6, 4) {};
		\node [style={vert_black}] (3) at (-3.5, -4) {};
		\node [style={vert_black}] (4) at (3.5, -4) {};
		\node [style={vert_black}] (5) at (0, 6) {};
		\node [style={vert_black}] (6) at (-3.75, 3.5) {};
		\node [style={vert_black}] (7) at (-2.5, -1.25) {};
		\node [style={vert_black}] (8) at (2.5, -1.25) {};
		\node [style={vert_black}] (9) at (3.75, 3.5) {};
		\node [style={vert_dest}] (10) at (0, 9) {};
		\node [style={vert_dest}] (11) at (6, 4) {};
		\node [style={vert_dest}] (12) at (-3.75, 3.5) {};
		\node [style={vert_dest}] (13) at (-2.5, -1.25) {};
		\node [style={vert_dest}] (14) at (2.5, -1.25) {};
		\node [style={vert_dest}] (15) at (-3.5, -4) {};
	
		\draw (0) to (5);
		\draw (5) to (8);
		\draw (8) to (6);
		\draw (6) to (9);
		\draw (9) to (7);
		\draw (7) to (5);
		\draw (0) to (2);
		\draw (2) to (9);
		\draw (0) to (1);
		\draw (1) to (6);
		\draw (1) to (3);
		\draw (3) to (4);
		\draw (4) to (2);
		\draw (8) to (4);
		\draw (7) to (3);
	
\end{tikzpicture}
\caption{The vertices indicated with a red circle form a 
2-distance mutual visibility set with maximum cardinal.}\label{fig:peter}
\end{figure}
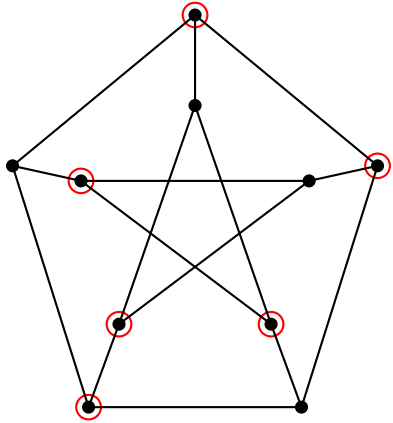

\begin{theorem}\label{eveng}
    Let $G$ be a graph with $n$ vertices and even girth $g\ge 6$. If $S$ is a largest $g/2$-distance 
    mutual-visibility set of $G$ and $G[S]$ contains some edge, then 
    $$
        \displaystyle |S|\le n-\left\lceil\frac{\sqrt{(g-2)^2(g-4)^2+(4n-8)(g-2)(g-4)}-(g-2)(g-4)}{2}\right\rceil.
    $$
\end{theorem}

\begin{proof}
Let $S$ be a largest $g/2$-distance mutual-visibility set of $G$. Observe that there is not a path with vertices $u,v,w$ in $G$ 
such that $\{u,v,w\}\subseteq S$, because in such a case, from this path and a shortest $u,w$-path of length at most $g/2$ with the internal vertices not in $S$, we can find a cycle of length at most $g/2+2<g$, due to $g\ge 6$. 
Thus, as the induced subgraph $G[S]$ contains some edge, we deduce that $G[S]$ is formed by some isolated edges and some isolated vertices.

Consider any edge $uv\in E(G)$, with $\{u,v\}\subseteq S$. Then every vertex $z\in S\setminus\{u,v\}$ must be at 
distance $g/2$ from at least one of vertices $u$ and $v$ and at distance at least $g/2-1$ from the other one. That 
is, for every $z\in S\setminus\{u,v\}$, there is a shortest $z,u$-path and a shortest $z,v$-path, one of them of length $g/2$ and the 
other one of length at least $g/2-1$, with the internal vertices in $V(G)\setminus S$. Thus, the internal elements 
of the $z,u$-path and the $z,v$-path are either a path with $g/2-2$ vertices and a path with $g/2-1$ vertices, or two 
paths with $g/2-1$ vertices. Denote by $P_z$ the pair of internal elements of the shortest $z,u$-path and the shortest $z,v$-path. 
If $P_z=P_w$ for two vertices $z,w\in S\setminus\{u,v\}$, then a cycle of length 4 appears, and this is a contradiction.
Therefore, in the most favorable case, for every $z\in S\setminus\{u,v\}$ we must find (in $V(G)\setminus S$) a path 
with $g/2-2$ vertices and a path with $g/2-1$ vertices, and this pair of elements must be different for each vertex 
$z\in S\setminus\{u,v\}$. In other words, if we denote by $r$ the number of paths with $g/2-1$ vertices, and by $t$ 
the number of paths with $g/2-2$ vertices in $V(G)\setminus S$, which may be internal vertices in the shortest $z,u$-path and in the shortest $z,v$-path, for each $z\in S\setminus\{u,v\}$, then the number of different $P_z$ which we can find is $rt$. 
Thus, the inequality $|S|-2\le rt$ must hold.

As a consequence, we need to maximize $f(r,t)=rt$ with restriction $(g/2-1)r+(g/2-2)t\le n-|S|$. The solution to this 
problem is $r=(n-|S|)/(g-2)$ and $t=(n-|S|)/(g-4)$, so that $\displaystyle |S|-2\le \frac{n-|S|}{g-2}\cdot
\frac{n-|S|}{g-4}$, yielding that $$\displaystyle |S|\le n-\left\lceil\frac{\sqrt{(g-2)^2(g-4)^2+(4n-8)(g-2)(g-4)}-
(g-2)(g-4)}{2}\right\rceil.$$
\end{proof}

\section{Concluding remarks}
We have initiated in this article the study of a ``limited'' version of the mutual-visibility problem for graphs, where such limitation is understood in the sense of the length of the paths used while searching for the visibility properties. We next include some possible research lines that might be of interest as a continuation of this investigation.
\begin{itemize}
    \item Find structural properties of the graphs $G$ of diameter $d$ satisfying that $\mu_k(G)\ge \max\{\mu(H)\,:\,H\in\mathcal{G}^k\}$ for some $k\in [d]$.

    \item Propositions \ref{prop:strong-P_r-P_2} and \ref{prop:strong-P_r-K_s} suggest studying the $k$-distance mutual-visibility number of other families of strong product graphs, as well as, the lexicographic product case, since indeed the strong product $P_r\boxtimes K_s$ can be also seen as the lexicographic product $P_r\circ K_s$.
    
    \item To characterize the extremal graphs attaining the upper bounds given by Theorems \ref{oddg} and \ref{eveng} is still open.

    \item Studying the $k$-distance mutual-visibility number of grid graphs $P_r\Box P_t$. Is it the case that such graphs satisfy the equality in the bound of Theorem \ref{th:subgraphs} (for at least large enough values of $r$ and $t$)?

    \item Theorem \ref{th:corona} suggest considering the study of the $k$-distance mutual-visibility number of corona graphs in general.
\end{itemize}

\section*{Acknowledgments}

I.\ G.\ Yero has been partially supported by the Spanish Ministry of Science and Innovation through the grant 
PID2019-105824GB-I00. Moreover, this investigation was made while this author (I.G. Yero) was making a temporary 
stay at ``Universitat Rovira i Virgili'' supported by the program ``Ayudas para la recualificaci\'on del sistema 
universitario espa\~{n}ol para 2021-2023, en el marco del Real Decreto 289/2021, de 20 de abril de 2021''.

\section*{Author contributions statement} 
All authors contributed equally to this work.

\section*{Conflicts of interest} 
The authors declare no conflict of interest.

\section*{Data availability} 
No data was used in this investigation.

\end{document}